\documentclass{article}
\usepackage{graphicx} 
\usepackage{amsmath, amssymb, amsthm, mathrsfs}

\usepackage{fullpage}
\usepackage{bbm}
\usepackage{tikz}
\usepackage{stmaryrd}
\usepackage{mathtools}
\usepackage{array}
\usepackage[colorlinks=true,
linkcolor=blue,citecolor=blue,
urlcolor=blue]{hyperref}

\allowdisplaybreaks

\newtheorem{theorem}{Theorem}
\newtheorem{lemma}[theorem]{Lemma}

\newtheorem{proposition}[theorem]{Proposition}
\theoremstyle{definition}

\theoremstyle{remark}

\newcommand{\cF}{\mathcal{F}}
\newcommand{\cH}{\mathcal{H}}
\newcommand{\cA}{\mathcal{A}}
\newcommand{\cB}{\mathcal{B}}
\newcommand{\cI}{\mathcal{I}}

\newcommand{\cE}{\mathcal{E}}
\newcommand{\cC}{\mathcal{C}}

\DeclareMathOperator{\co}{co_2}

\DeclareMathOperator{\cojp}{co_{j, p}}
\DeclareMathOperator{\coonep}{co_{1, p}}
\DeclareMathOperator{\cotwop}{co_{2, p}}

\title{The Hilton-Milner theorem in the $\ell_2$-norm}
\author{George Brooks\thanks{University of South Carolina, Columbia, SC, USA. ({\tt ghbrooks@email.sc.edu}). Partially supported by NSF RTG Grant DMS 2038080.}
\and
William Linz\thanks{University of Memphis, Memphis, TN, USA. ({\tt wlinz1@memphis.edu}).}}
\date{\today}

\begin{document}

\maketitle

\begin{abstract}
The Hilton-Milner theorem is a classical theorem in extremal set theory which determines the maximum size of a nontrivial $k$-uniform intersecting family. In this paper, we prove an exact version of the Hilton-Milner theorem in the $\ell_2$-norm. Our method also gives an asymptotic version of the Hilton-Milner theorem in the $(j, p)$-norm.     
\end{abstract}

\section{Introduction}

For integers $a < b$, we define $[a, b]:= \{x\in \mathbb{Z}: a\le x \le b\}$ and set $[n]:= [1, n]$ for a positive integer $n$. For a set $V$ and a positive integer $k$, we use the notation $\binom{V}{k}$ to denote the family of all $k$-element subsets of $V$. 

A family of sets $\cF \subset \binom{[n]}{k}$ is \emph{intersecting} if $|F\cap F'| \ge 1$ for every $F, F'\in \cF$. The maximum size of a $k$-uniform intersecting family is a classic result in extremal set theory which was proved by Erd\H{o}s, Ko and Rado~\cite{EKR}. 

\begin{theorem}[Erd\H{o}s-Ko-Rado]\label{thm:ekr}
Let $n$ and $k$ be positive integers and let $\cF \subset \binom{[n]}{k}$ be an intersecting family. Then for $n\ge 2k$,
\[|\cF|\le \binom{n-1}{k-1}.\]
For $n > 2k$, equality holds only if $\cF \cong\{F: F\in\binom{[n]}{k},\, 1\in F\}$. 
\end{theorem}

The Erd\H{o}s-Ko-Rado theorem has been generalized and extended in many different ways. See \cite{Fra87, FT2016} for general surveys on extremal set theory and the book of Godsil and Meagher~\cite{GM} for a general survey on algebraic analogues of the Erd\H{o}s-Ko-Rado theorem. 

Recently, Balogh, Clemen and Lidick\'y~\cite{BCL} introduced a new Tur\'an-type problem using the codegree squared sum as the measure of the size of a hypergraph. For a $k$-uniform hypergraph $\cH = ([n], \cE)$, define the \emph{codegree} of a set $J\subset [n]$ to be $d_{\cH}(J):= |\{E: J\subset E \in \cE\}|$. When it is clear from context, we may omit the subscript and simply write $d(J)$. The \emph{codegree vector} of $\cH$ is the vector ${\bf v} \in \mathbb{Z}^{\binom{[n]}{k-1}}$ whose entries are the codegrees of the $(k-1)$-element subsets of $[n]$. The \emph{codegree squared sum} \[\co(\cH) := \sum_{J \in \binom{[n]}{k-1}}d(J)^2\] is defined as the square of the $\ell_2$-norm of the codegree vector ${\bf v}$. Note that $\sum_{J\in \binom{[n]}{k-1}}d(J) = k|\cH|$, so hypergraph Tur\'an problems in the $\ell_1$-norm of the codegree vector are essentially equivalent to classical hypergraph Tur\'an problems.   

In a previous paper~\cite{BL}, we proved a version of the Erd\H{o}s-Ko-Rado theorem in the $\ell_2$-norm by showing that the stars also have maximum codegree squared sum among intersecting families.  

\begin{theorem}[Erd\H{o}s-Ko-Rado in $\ell_2$-norm]\label{thm:ekrl2}
Let $\cF \subset \binom{[n]}{k}$ be an intersecting family with $n\ge 2k$. Then,
\[\co(\cF) \le \binom{n-1}{k-1}(1+(n-k+1)(k-1)),\]
with equality only if $\cF \cong \{F\subset \binom{[n]}{k} : 1\in F\}$ if $n > 2k$.  If $n=2k$, equality holds only for the $1$-star $\cF \cong \{F\subset \binom{[2k]}{k} : 1\in F\}$ and the complement of the $1$-star $\cF \cong\ \binom{[2k-1]}{k}$. 
\end{theorem}

The result is a direct consequence of an inequality of Bey~\cite{Bey}.

\begin{theorem}[Bey's inequality]\label{thm:bey}
Let $\cH = (V, \cE)$ be a $k$-uniform hypergraph with $|V| = n$. 
Then, 
\[\sum_{J\in \binom{V}{k-1}}d(J)^2\le \frac{k}{\binom{n-1}{k-1}}|\cE|^2 + (k-1)(n-k)|\cE|.\]

Equality holds if $\cH$ is the $1$-star \[S_k^1 = \left\{F\in \binom{[n]}{k}: 1 \in F\right\}\] or its complement; the complete graph $K_k^n$ or its complement; or if $n=k+1$ and $\cH$ is the $t$-star $S_k^t = \{F \in \binom{[k+1]}{k}: [t] \subset F\}$ for $2\le t\le \lfloor{\frac{k+1}{2}\rfloor}$ or its complement.
\end{theorem} 

We conjectured in \cite{BL} that a version of the Hilton-Milner theorem should also hold in the $\ell_2$-norm. Hilton and Milner~\cite{HM} determined the maximum size of a nontrivial $k$-uniform intersecting family and completely characterized the equality cases. A family of sets $\cF$ is \emph{nontrivial intersecting} if $\cF$ is intersecting and $\bigcap_{F\in\cF}F = \emptyset$. The Hilton-Milner theorem was the first sharpening of the Erd\H{o}s-Ko-Rado theorem. 

\begin{theorem}[Hilton-Milner]\label{thm:hm}
Let $n$ and $k$ be positive integers and let $\cF \subset \binom{[n]}{k}$ be a nontrivial intersecting family. Then, for $n > 2k$, 
\[|\cF|\le \binom{n-1}{k-1} - \binom{n-k-1}{k-1} + 1.\]
If $k\ge 4$, equality holds only if $\cF \cong \cH(n, k):=\{F\in \binom{[n]}{k}: 1\in F,\, F\cap [2, k+1]\neq\emptyset\} \cup \{[2, k+1]\}$. If $k=3$, equality holds only if $\cF\cong\cH(n, 3)$ or $\cF\cong \cA(n, 3):=\{F\in \binom{[n]}{3}: |F\cap [3]| \ge 2\}.$
\end{theorem}

Like the original Erd\H{o}s-Ko-Rado theorem, the theorem of Hilton and Milner has been extended and generalized in many different ways. For example, versions of the Hilton-Milner theorem have been proven for nontrivial $t$-intersecting, $d$-wise intersecting, and $d$-wise $t$-intersecting families \cite{AK2, OV, BaL}, subspaces~\cite{BBCFMPS, CLWZ, WXY}, and the bilinear forms graph~\cite{GLW, CLW}. 

In this paper, we prove that the families $\cH(n, k)$ and $\cA(n, k)$ (for $k=3$) also have maximum codegree squared sum among nontrivial intersecting $k$-uniform families.

\begin{theorem}[Hilton-Milner in $\ell_2$-norm]\label{thm:hml2}
Let $\cF \subset \binom{[n]}{k}$ be a nontrivial intersecting family. Then, for $k\ge 3$ and $n > 2k$, 
\[\co(\cF) \le \co(\cH(n, k)),\]
with equality if and only if $\cF \cong \cH(n, k)$ if $k\ge 4$, and if and only if $\cF \cong \cH(n, 3)$ or $\cF\cong \cA(n, 3)$ if $k=3$.

If $n=2k$, then 
\[\co(\cF) \le k^2\binom{2k-1}{k-1},\]
and equality holds only if $\cF \cong \binom{[2k-1]}{k}$. 
\end{theorem} 

We have previously proven that Theorem~\ref{thm:hml2} is true when $n=2k$ (by Theorem~\ref{thm:ekrl2}) and when $k=3$~\cite{BL}. Wu and Zhang~\cite{WZ} proved an $\ell_2$-norm version of the $t$-intersecting Hilton-Milner theorem of Ahlswede and Khachatrian~\cite{AK2} for $t\ge 2$ and stated that their argument also proves Theorem~\ref{thm:hml2} for $n\ge 3k$. Although we therefore only need to cover the range $2k < n < 3k$ to resolve the unknown cases of Theorem~\ref{thm:hml2}, we give a proof which covers the entire range $n > 2k$. One reason is that our proof can be generalized to obtain a version of the Hilton-Milner theorem in the $(j, p)$-norm.

The \emph{$(j, p)$-norm} of a hypergraph $\cF$ is a generalization of the $\ell_2$-norm that was first systematically studied by Chen, Il'kovi\v{c}, Le\'on, Liu, and Pikhurko~\cite{CILLP}. For a hypergraph $\cF$, the $(j, p)$-norm is defined as the $p$th power of the $\ell_p$-norm of the codegree vector whose entries are the codegrees of the sets $J \in \binom{[n]}{j}$
\[\cojp(\cF) := \sum_{J \in \binom{[n]}{j}} d(J)^p.\]
Recently, there has been quite a lot of interest in finding $(j, p)$-norm versions of classical Tur\'an-type results. Within extremal set theory, Zhao and Wang~\cite{ZW} and Cao, Lu and Zhang~\cite{CLZ} have independently proven versions of the Erd\H{o}s-Ko-Rado theorem in the $(j, p)$-norm, extending Theorem~\ref{thm:ekrl2}. 

We now give our version of the Hilton-Milner theorem in the $(j, p)$-norm. 

\begin{theorem}\label{thm:hmjp}
    Let $k\ge 4$ be an integer, let $1\le j\le k-1$ be an integer, and let $p > 1$ be a real number. There exists $n_0 = n_0(k, p, j)$ such that for any nontrivial intersecting family $\cF \subset \binom{[n]}{k}$ with $n\ge n_0$, we have 
    \[\cojp(\cF) \le \cojp(\cH(n, k)).\]
    Equality holds if and only if $\cF \cong \cH(n, k).$
\end{theorem}

We can also prove an exact result for nontrivial intersecting $3$-uniform families. 

\begin{proposition}\label{thm:jpk=3}
    If $\cF \subset \binom{[n]}{3}$ is a nontrivial intersecting family, then for all $n\ge 7$ and $j\in \{1, 2\}$ and $p > 1$,
    \[\cojp(\cF) \le \begin{cases} \cojp(\cA(n, 3)) & \text{if } 1 < p < 2\\ \cojp(\cH(n, 3))& \text{if } p\ge 2.\end{cases}\]
    For $1 < p < 2$, equality holds only if $\cF \cong \cA(n, 3)$; for $p=2$, equality holds only if $\cF \cong \cA(n, 3)$ or $\cF \cong \cH(n, 3)$; for $p > 2$, equality holds only if $\cF \cong \cH(n, 3)$. 
\end{proposition}

\section{Proof of Hilton-Milner theorem in \texorpdfstring{$\ell_2$-norm}{}}

We first record some useful lemmas, and then give the proof of Theorem~\ref{thm:hml2}. 

\subsection{Preliminaries}

We give an exact formula for the codegree squared sum of the Hilton-Milner family $\cH(n, k)$. One can determine $\co(\cH(n, k))$ by determining the codegree of each set in $\binom{[n]}{k-1}$ and then performing tedious algebraic manipulation. In the sequel, we opt for a more combinatorial approach to determine $\co(\cH(n, k))$. 

Let $\cF \subset \binom{[n]}{k}$ be an intersecting family. For $i\in [n]$, the set $\cF(i):= \{F\in \cF: \, i\in F\}$ is the family of all sets in $\cF$ containing $i$. The maximum degree of the family $\cF$ is $\Delta(\cF):= \max_{1\le i\le n}|\cF(i)|$. The \emph{diversity} of the family $\cF$ is the quantity $\gamma(\cF):= |\cF| - \Delta(\cF)$. After relabeling, we can assume that the maximum degree vertex in $\cF$ is $1$. Let $\cA = \{F \setminus \{1\}: 1\in F\in \cF\} \subseteq \binom{[2, n]}{k-1}$ and $\cB = \{F\in \cF: 1\notin F\} \subseteq \binom{[2, n]}{k}$. Observe that $\cA$ and $\cB$ are cross-intersecting. Let us also note that $\gamma(\cF) = 0$ if and only if $\cF$ is a star and $\gamma(\cF) = 1$ if and only if $\cF$ is nontrivial and $\cF \subset \cH(n, k)$. Indeed, if $\gamma(\cF) = 1$, then let $B$ be the unique set in $\cB$. Every set of the form $\{1\} \cup A$, $A\in \cA$ must intersect $B$, proving that $\cF \subset \cH(n, k)$ (after relabeling the elements of $B$). Let $\cI(\cA, \cB):= \{(A, B): A\in \cA,\, B\in \cB,\, A\subset B\}$. 

 The \emph{Johnson graph} $J(n, k)$ is the graph with vertex set $V(J(n, k)) = \binom{[n]}{k}$ and edge-set $E(J(n, k)) = \{AB: |A\cap B| = k-1\}$. Henceforth, for integers $n$ and $k$ and a family $\cF \subseteq \binom{[n]}{k}$, we use the notation $e_{J(n, k)}(\cF)$ to denote the number of edges in the induced subgraph $J(n, k)[\cF]$ of the Johnson graph $J(n, k)$. 

We recall an alternate, useful expression for the codegree squared sum~\cite{BL}.

\begin{lemma}[\protect{\cite[Lemma 2.1]{BL}}]\label{lem:reform}
For any $\cF \subseteq \binom{[n]}{k}$,
\begin{align*}
\co(\cF) &= k|\cF|+ 2 \left|\left\{\{F, F'\} \in \binom{\cF}{2} : |F\cap F'| = k -1 \right\}\right|. \\
&=k|\cF| + 2e_{J(n, k)}(\cF). 
\end{align*}
\end{lemma}

Lemma~\ref{lem:reform} is a consequence of summing the identity $d(J)^2 = d(J) + 2\binom{d(J)}{2}$ over all sets $J \in \binom{[n]}{k-1}$.

We use Lemma~\ref{lem:reform} to give an alternative formulation of $\co(\cF)$ based on this decomposition.

\begin{lemma}\label{lem:divdecomp}
For any intersecting family $\cF \subset \binom{[n]}{k}$,
\begin{align*}
    \co(\cF) &= k|\cA| + k|\cB| + 2e_{J(n-1, k-1)}(\cA) + 2e_{J(n-1, k)}(\cB) + 2|\cI(\cA, \cB)|\\
    &\le k|\cA| + 3k|\cB| +2e_{J(n-1, k-1)}(\cA) + 2e_{J(n-1, k)}(\cB)
\end{align*}
\end{lemma}
\begin{proof}[Proof of Lemma~\ref{lem:divdecomp}]
The edges in $J(n, k)[\cF]$ are of three types: (1) edges where $1$ is contained in both endpoints; (2) edges where both endpoints are contained in $\cB$; (3) crossing edges, with one endpoint containing $1$. The first type is in one-to-one correspondence with the edges in $J(n - 1, k-1)[\cA]$, and the second type is in one-to-one correspondence with $J(n-1, k)[\cB]$. The third type is in one-to-one correspondence with the pairs in $\cI(\cA, \cB)$. The stated equality then follows from Lemma~\ref{lem:reform}. 

The second statement follows from the observation that $|\cI(\cA, \cB)| \le k |\cB|$, as each $B\in \cB$ has only $k$ subsets of size $k-1$. 
\end{proof}

We compute the codegree squared sum of the Hilton-Milner family $\cH(n, k)$ using Lemma~\ref{lem:divdecomp}. 

    \begin{lemma}[Codegree squared sum of $\cH(n, k)$]\label{lem:hmcdsum}
For $n\ge 2k$,
\[\co(\cH(n, k)) = (k+(k-1)(n-k))\binom{n-1}{k-1} - (k + (k-1)(n-k)+k(k-1))\binom{n-k-1}{k-1} + 3k. \]
\end{lemma}
\begin{proof}
  For the Hilton-Milner family $\cH(n, k)$, we have that $\cA = \{F\in \binom{[2, n]}{k-1}: F\cap [2, k+1] \neq\emptyset\}$ and $\cB = \{[2, k+1]\}.$ We have $|\cA| = \binom{n-1}{k-1} - \binom{n-k-1}{k-1}$ and $|\cB| = 1$, so that $2e_{J(n-1, k)}(\cB) = 0$. Furthermore, $|\cI(\cA, \cB)| = k$, since for $2\le i\le k+1$, each set $[2, k+1]\setminus \{i\}$ is in $\cA$. We lastly need to compute $2e_{J(n-1, k-1)}(\cA)$. There are $\frac12(k-1)(n-k)\binom{n-1}{k-1}$ total edges in $J(n-1, k-1)$ and we need to remove the $\frac12(k-1)(n-2k)\binom{n-k-1}{k-1}$ edges in the subgraph of $J(n-1, k-1)$ induced by $\binom{[k+2, n]}{k-1}$ and the $k(k-1)\binom{n-k-1}{k-1}$ crossing edges with one endpoint in $\cA$ and one endpoint in $\binom{[k+2, n]}{k-1}$. This gives 
  \begin{align*}
      e_{J(n-1, k-1)}(\cA) &= \frac12(k-1)(n-k)\binom{n-1}{k-1} - \frac12(k-1)(n-2k)\binom{n-k-1}{k-1} - (k-1)k\binom{n-k-1}{k-1}\\
      &=\frac12(k-1)(n-k)\binom{n-1}{k-1} - \left(\frac12(k-1)(n-k) + \frac12k(k-1)\right)\binom{n-k-1}{k-1}.
      \end{align*}
Lemma~\ref{lem:divdecomp} implies the desired result.   
\end{proof}

We also need the following general result of Hilton and Milner~\cite[Theorem 3]{HM} extending their well-known theorem. This theorem was later reproved by Han and Kohayakawa~\cite{HK}, who also proved the case $k=3$ and determined the families that attain equality. 

\begin{theorem}\label{thm:hm2}
	Let $\cF\subset \binom{[n]}{k}$ be an intersecting family such that $\gamma(\cF) \geq s$ with $k \geq \min\{3, s+1\}$ and $n \geq 2k$. Then,
	\[
	|\cF| \leq \begin{dcases}
		\binom{n-1}{k-1} - \binom{n-k}{k-1} + n - k & \text{if }2 < k \leq s+2,\\
		\binom{n-1}{k-1} - \binom{n-k}{k-1} + \binom{n-k-s}{k-s-1} + s & \text{if }k\leq 2\text{ or }k\geq s+2.
	\end{dcases}
	\]
\end{theorem}

\subsection{Proof of Theorem~\ref{thm:hml2}}
We first prove a stability result for the family $\cH(n, k)$ in the $\ell_2$-norm. 

\begin{theorem}\label{thm:completeHMstability}
	Suppose $\cF \subset \binom{[n]}{k}$ is a nontrivial intersecting family with $n > 2k \geq 8$ and $\cF$ is not contained in any family isomorphic to $\cH(n, k)$. Then
	\[ \co(\cH(n, k)) - \co(\cF) \ge \max\{f(n, k), g(n, k)\},\]
	where
    \[f(n, k) = (k-1)\left(n\binom{n-k-2}{k-2}-\binom{n-1}{k-1}\right)+(2k-3)\left(\binom{n-k-2}{k-2} - k\right) \]
    and 
    \[g(n, k) = \frac{n+3k^2-5k}{k-1}\binom{n-k-2}{k-2} -2(k-1)(n-k) -k - \frac{k}{\binom{n-1}{k-1}}\left[\frac{n-2}{k-1}\binom{n-k-2}{k-2}-2\right]^2 \]
\end{theorem}
\begin{proof}
	We may assume that $\cF$ is an inclusion-maximal intersecting family.  Define
	
	\[
	\mathcal{C}_{k-1} = \left\{ J \in \binom{[n]}{k-1} : J\text{ is a cover of }\cF \right\}.
	\]
	
	We first prove the family $\mathcal{C}_{k-1}$ is intersecting. If $|\mathcal{C}_{k-1}| \leq 1$ then $\mathcal{C}_{k-1}$ is trivially intersecting. Suppose $\mathcal{C}_{k-1}$ is not intersecting and $|\mathcal{C}_{k-1}| > 1$. Let $I,J \in \mathcal{C}_{k-1}$ such that $I \cap J = \emptyset$. Since $n > 2k$ we may choose distinct $s, t \in [n]\setminus (I\cup J)$. The sets $I\cup\{s\}$ and $J\cup\{t\}$ intersect every $F \in \cF$ so the maximality of $\cF$ implies $I\cup\{s\}, J\cup\{t\} \in \cF$. Hence, $I\cap J \neq \emptyset$, a contradiction.
	
	If $\mathcal{C}_{k-1}$ is nontrivial, then by Theorem~\ref{thm:hm},
	\begin{equation}\label{eq:nontrivial}
	|\mathcal{C}_{k-1}| \leq \binom{n-1}{k-2} - \binom{n-k}{k-2} + 1.
	\end{equation}
	
	On the other hand, suppose that $\mathcal{C}_{k-1}$ has an element $i$ which is contained in all sets in $\mathcal{C}_{k-1}$. As $\cF$ is nontrivial and $\gamma(\cF) \geq 2$, there exist distinct $F, F' \in \cF$ such that $i \not \in F \cup F'$. Every set in $\mathcal{C}_{k-1}$ is a cover of $\cF$ which implies $(J-i)\cap F$ and $(J-i)\cap F'$ are nonempty for all $J \in \mathcal{C}_{k-1}$. We count the number of possible sets in $\binom{[n]\setminus\{i\}}{k-2}$ that intersect both $F$ and $F'$, giving an upper bound on $|\mathcal{C}_{k-1}|$,
	\begin{equation}\label{eq:trivial}
	|\mathcal{C}_{k-1}| \leq \binom{n-1}{k-2} - 2\binom{n - k - 1}{k-2} + \binom{n-2k+|F\cap F'|-1}{k-2}.
	\end{equation}
	Since $1\leq |F\cap F'| \leq k-1$, the upper bound in \eqref{eq:trivial} is increasing in $|F\cap F'|$ and the maximum occurs when $|F\cap F'| = k-1$. The upper bound in \eqref{eq:trivial} is at least the upper bound in \eqref{eq:nontrivial} when $|F\cap F'| = k-1$. Hence
	\begin{equation*}
		|\mathcal{C}_{k-1}| \leq \binom{n-1}{k-2} - 2\binom{n - k - 1}{k-2} + \binom{n-k-2}{k-2}.
	\end{equation*}
	
	Now, every $J \in \binom{[n]}{k-1}$ with $d(J) \geq k+1$ is in $\mathcal{C}_{k-1}$. Indeed, if $J$ is not a cover then $J \cap F' = \emptyset$ for some $F' \in \cF$ and every $F \in \cF$ containing $J$ is of the form $J \cup \{i\}$ where $i \in F'$ as $\cF$ is intersecting and therefore $d(J) \leq k$. Hence, all $(k-1)$-sets with codegree at least $k+1$ are covers and so
	\begin{equation}\label{eq:dTbound}
		\left|\left\{J \in \binom{[n]}{k-1}: d(J) \geq k+1\right\}\right| \leq |\mathcal{C}_{k-1}| \leq \binom{n-1}{k-2} - 2\binom{n - k - 1}{k-2} + \binom{n-k-2}{k-2}.
	\end{equation}
	
	Since $\gamma(\cF) \geq 2$, by Theorem~\ref{thm:hm2}, we have
	\begin{equation}\label{eq:genHM}
		|\cF| \leq \binom{n-1}{k-1} - \binom{n-k}{k-1} + \binom{n-k-2}{k-3} + 2 = \binom{n-1}{k-1} - \binom{n-k-1}{k-1} - \binom{n-k-2}{k-2} + 2.
	\end{equation}
	
	Since $d(J) \leq n - k+1$, it follows that
	\begin{align*}
		\co(\cF) &= \sum_{J \in \binom{[n]}{k-1}} d(J)^2\\
		&= k\sum_{J \in \binom{[n]}{k-1}}d(J) + \sum_{J \in \binom{[n]}{k-1}}d(J)(d(J)-k)\\
		&\leq k^2|\cF| + \sum_{\substack{J \in \binom{[n]}{k-1}\\ d(J) \geq k+1}}d(J)(d(J)-k)\\
		&\leq k^2 |\cF| + |\mathcal{C}_{k-1}|(n-k+1)(n-2k+1)
	\end{align*}
	
	Using \eqref{eq:dTbound} and \eqref{eq:genHM} we have
	\begin{align*}
		\co(\cH(n,k)) - \co(\cF)&\geq \co(\cH(n,k)) - k^2 \left[\binom{n-1}{k-1} - \binom{n-k-1}{k-1} - \binom{n-k-2}{k-2}+2\right]\\
		&\quad\quad - (n-k+1)(n-2k+1) \left[\binom{n-1}{k-2}-2\binom{n-k-1}{k-2}+\binom{n-k-2}{k-2}\right]\\
		&= ((k-1)n + 2k-3)\binom{n-k-2}{k-2} - (k-1)\binom{n-1}{k-1}-k(2k-3)\\
        &= (k-1)\left(n\binom{n-k-2}{k-2}-\binom{n-1}{k-1}\right)+(2k-3)\left(\binom{n-k-2}{k-2} - k\right).
	\end{align*}
    \newpage
    We provide the other lower bound using Bey's inequality,
    \begin{align*}
        \co(\cH(n,k)) - \co(\cF)&\geq \co(\cH(n,k)) - \frac{k}{\binom{n-1}{k-1}}|\cF|^2 -(k-1)(n-k)|\cF|\\
        &\geq (k+(k-1)(n-k))\binom{n-1}{k-1} - (k + (k-1)(n-k)+k(k-1))\binom{n-k-1}{k-1} + 3k\\
        &\quad\quad -\frac{k}{\binom{n-1}{k-1}} \left[\binom{n-1}{k-1} - \binom{n-k-1}{k-1} - \binom{n-k-2}{k-2} + 2\right]^2 \\
        &\quad\quad - (k-1)(n-k) \left[\binom{n-1}{k-1} - \binom{n-k-1}{k-1} - \binom{n-k-2}{k-2} + 2\right] \\
        &= k\binom{n-1}{k-1} - k^2 \binom{n-k-1}{k-1} + 3k+ (k-1)(n-k)\binom{n-k-2}{k-2} - 2(k-1)(n-k)\\
        &\quad\quad - k\binom{n-1}{k-1} + 2k \left[\binom{n-k-1}{k-1}+\binom{n-k-2}{k-2} -2\right]\\
        &\quad\quad -\frac{k}{\binom{n-1}{k-1}}\left[\binom{n-k-1}{k-1}+\binom{n-k-2}{k-2}-2\right]^2\\
        &=  (2k - k^2) \binom{n-k-1}{k-1} + ((k-1)(n-k)+2k)\binom{n-k-2}{k-2} - 2(k-1)(n-k)-k\\
        &\quad\quad -\frac{k}{\binom{n-1}{k-1}}\left[\binom{n-k-1}{k-1}+\binom{n-k-2}{k-2}-2\right]^2\\
        &= \frac{n+3k^2-5k}{k-1}\binom{n-k-2}{k-2} -2(k-1)(n-k) -k - \frac{k}{\binom{n-1}{k-1}}\left[\frac{n-2}{k-1}\binom{n-k-2}{k-2}-2\right]^2.
    \end{align*}

    \end{proof}
    Using Theorem~\ref{thm:completeHMstability}, we can prove the Hilton-Milner theorem in the $\ell_2$-norm. 
    
    \begin{proof}[Proof of Theorem~\ref{thm:hml2}]

    If $\cF \subset \binom{[n]}{k}$ is a proper subfamily of a family isomorphic to $\cH(n, k)$, then strict monotonicity of the $(j, p)$-norm under edge addition gives $\co(\cF) < \co(\cH(n, k))$. Suppose $\cF$ is a nontrivial intersecting family which is not a subfamily of any copy of $\cH(n, k)$. As the cases $n = 2k$ and $k=3$ of Theorem~\ref{thm:hml2} were resolved in \cite{BL}, we assume $n > 2k \ge 8$. By Theorem~\ref{thm:completeHMstability},

    \[\co(\cH(n, k)) - \co(\cF) \ge \max\{f(n, k), g(n, k)\},\]
    where
   \[
    f(n,k) = (k-1)\left(n\binom{n-k-2}{k-2}-\binom{n-1}{k-1}\right)+(2k-3)\left(\binom{n-k-2}{k-2} - k\right)
    \]
    and
    \[
    g(n,k) = \frac{n+3k^2-5k}{k-1}\binom{n-k-2}{k-2} -2(k-1)(n-k) -k - \frac{k}{\binom{n-1}{k-1}}\left[\frac{n-2}{k-1}\binom{n-k-2}{k-2}-2\right]^2
    \]

    We will show that $f(n,k) > 0$ or $g(n,k) > 0$ whenever $8\le 2k < n$. We first dispense with the case that $n = 2k+1$. Consider
    \begin{align*}
        g(2k+1, k) &= \frac{2k+1+3k^2-5k}{k-1}\binom{k-1}{k-2} -2(k-1)(k+1) -k - \frac{k}{\binom{2k}{k-1}}\left[\frac{2k-1}{k-1}\binom{k-1}{k-2}-2\right]^2\\
        &= (k-1)(k-3) -\frac{k(2k-3)^2}{\binom{2k}{k-1}}.
    \end{align*}

    Since $k \geq 4$ we have $\binom{2k}{k-1} \geq \binom{2k}{3}$. Hence
    \[
        g(2k+1,k) \geq (k-1)(k-3) - \frac{3(2k-3)^2}{(2k-1)(2k-2)}>0.
    \]

    Therefore, we may assume $n\geq 2k+2$. It follows that $\binom{n-k-2}{k-2} \geq \binom{k}{k-2}$. Suppose that $f(n,k) \leq 0$. We will show that this implies $g(n,k) > 0$. We observe
    \begin{align*}
    (k-1)\left(n\binom{n-k-2}{k-2}-\binom{n-1}{k-1}\right) &\leq (2k-3)\left(k- \binom{n-k-2}{k-2}\right)\\
    &\leq (2k-3) \left(k - \binom{k}{k-2}\right) \\
    &= \frac{3k}{2}(1-k)(2k-3).
    \end{align*}

    Since $k \geq 4$, it follows that
    \begin{equation}\label{eq:binomrel}
    n\binom{n-k-2}{k-2}<\binom{n-1}{k-1}.
    \end{equation}

    The preceding inequality implies $n < k^2$. Suppose to the contrary that $n \geq k^2$. Consider
    \begin{align*}
        \frac{\binom{n-1}{k-1}}{\binom{n-k-2}{k-2}} &= \frac{(n-1)!}{(k-1)!(n-k)!}\cdot\frac{(k-2)!(n-2k)!}{(n-k-2)!}\\
        &= \frac{1}{k-1}\left(\frac{\prod_{i=1}^{k-1}(n-k+i)}{\prod_{j=1}^{k-2}(n-2k+j)}\right)\\
        &= \frac{n-1}{k-1}\left(\prod_{i=1}^{k-2} \frac{n-k+i}{n-2k+i}\right)\\
        &= \frac{n-1}{k-1}\left(\prod_{i=1}^{k-2} \left(1+\frac{k}{n-2k+i}\right)\right)\\
        &\leq \frac{n-1}{k-1}\left(1 + \frac{k}{n-2k+1}\right)^{k-2}\\
        &\leq \frac{n-1}{k-1}\left(1 + \frac{k}{(k-1)^2}\right)^{k-2}.
    \end{align*}

    Since $\frac{k}{(k-1)^2} < \frac{1}{k-2}$ and $k\geq 4$ we have
    \[
    \frac{\binom{n-1}{k-1}}{\binom{n-k-2}{k-2}} \leq \frac{n-1}{k-1}\left(1 + \frac{1}{k-2}\right)^{k-2} < \frac{e(n-1)}{k-1} < n.
    \]

    This contradicts \eqref{eq:binomrel}. Therefore we may assume $n < k^2$.

    We now show that $g(n, k ) > 0$ by a straightforward calculation utilizing \eqref{eq:binomrel} and the assumptions that $8 \leq 2k < n < k^2$.
    \newpage
    \begin{align*}
        g(n,k) &= \frac{n+3k^2-5k}{k-1}\binom{n-k-2}{k-2} -2(k-1)(n-k) -k - \frac{k}{\binom{n-1}{k-1}}\left[\frac{n-2}{k-1}\binom{n-k-2}{k-2}-2\right]^2\\
        &\geq \frac{n+3k^2-5k}{k-1}\binom{n-k-2}{k-2} -2(k-1)(n-k) -k - \frac{k}{\binom{n-1}{k-1}}\left[\frac{n-2}{k-1}\binom{n-k-2}{k-2}\right]^2\\
        &\geq \frac{n+3k^2-5k}{k-1}\binom{n-k-2}{k-2} -2(k-1)(n-k) -k - \frac{k}{n}\left(\frac{n-2}{k-1}\right)^2 \binom{n-k-2}{k-2}\\
        &= \binom{n-k-2}{k-2}\left(\frac{3k^3-8k^2+9k-n-\frac{4k}{n}}{(k-1)^2}\right)-2(k-1)(n-k) -k\\
        &> \binom{n-k-2}{k-2}\left(3k-4\right)-2(k-1)(n-k) -k\\
        &\geq \binom{n-k-2}{2}\left(3k-4\right)-2(k-1)(n-k) -k\\
        &> 0
    \end{align*}

    This completes the proof. 
\end{proof}

\section{Asymptotic Hilton-Milner theorem in \texorpdfstring{$(j, p)$-norm}{}}

In this section, we prove the following theorem, which implies Hilton-Milner for large $n$ in the $(j, p)$-norm.

\begin{theorem}\label{thm:hmjpnormlarge}
Let $k\ge 4$ be an integer and let $p > 1$ be a real number. Suppose $\cF \subset \binom{[n]}{k}$ is a nontrivial intersecting family which is not contained in any family isomorphic to $\cH(n, k)$. Then, as $n\rightarrow \infty$, uniformly over $\cF$ we have
\[\cojp(\cF)\le \begin{dcases} \left(\dfrac{k-1}{(j-2)!(k-j)!^p} + o(1)\right)n^{j-2 + p(k-j)} & \text{if } 2\le j\le k-1 \\ (\max\{(k-1)^p + k-1, 3\cdot 2^p\} + o(1))\binom{n-2}{k-2}^p & \text{if } j=1.\end{dcases}\]
\end{theorem}

Let us first show why Theorem~\ref{thm:hmjpnormlarge} implies Theorem~\ref{thm:hmjp}. We first record the $(j, p)$-norm of the Hilton-Milner family $\cH(n, k)$. 

\begin{lemma}\label{lem:mjpcalc}
Let $k$ and $j$ be fixed positive integers, let $p > 1$ be a real number and let $n\rightarrow \infty$. We have
\[\cojp(\cH(n, k)) = \begin{dcases}\left(\dfrac{k}{(j-2)!(k-j)!^p} + o(1)\right)n^{j-2+p(k-j)}& \text{if } 2\le j\le k-1 \\ (k^p + k +o(1))\binom{n-2}{k-2}^p & \text{if } j=1.\end{dcases}\]
\end{lemma}

\begin{proof}
We determine the codegree $d(J)$ in $\cH(n,k)$ for each $J \in \binom{[n]}{j}$ based on whether $1\in J$ as well as its interaction with the set $[2, k+1]$. Furthermore, we record the count of each type of $J \in \binom{[n]}{j}$.
\[
\renewcommand{\arraystretch}{2.25}
\begin{array}{ >{\displaystyle}c |>{\displaystyle}c | >{\displaystyle}c }
     \textit{Type} & \textit{Count} & d(J)\\
     \hline
     1 \in J,\, J \cap [2, k+1] \neq \emptyset & \binom{n-1}{j-1} - \binom{n-k-1}{j-1} & \binom{n-j}{k-j}\\
     1 \in J,\, J \cap [2, k+1] = \emptyset & \binom{n-k-1}{j-1} & \binom{n-j}{k-j} - \binom{n-k-j}{k-j}\\
     1 \not\in J,\, |J\cap [2,k+1]| = j & \binom{k}{j} & \binom{n-j-1}{k-j-1} + 1\\
     1 \not\in J,\, 1 \leq |J\cap[2,k+1]| \leq j-1 & \binom{n-1}{j} - \binom{n-k-1}{j}-\binom{k}{j} & \binom{n-j-1}{k-j-1}\\
     1 \not\in J,\, J \cap [2, k+1] = \emptyset & \binom{n-k-1}{j} & \binom{n-j-1}{k-j-1}-\binom{n-k-j-1}{k-j-1}
\end{array}
\]

For $j = 1$, two of the types are impossible. Hence, the $(1,p)$-norm of $\cH(n,k)$ is 
\begin{align*}
\coonep(\cH(n, k)) &= \left(\binom{n-1}{k-1} - \binom{n-k-1}{k-1}\right)^p + k\left(\binom{n-2}{k-2}+1\right)^p+(n-k-1)\left(\binom{n-2}{k-2} - \binom{n-k-2}{k-2}\right)^p\\
&=k^p\binom{n-2}{k-2}^p + o(n^{(k-2)p}) + k\left(\binom{n-2}{k-2}+1\right)^p + (n-k-1)\left(\binom{n-2}{k-2} - \binom{n-k-2}{k-2}\right)^p \\ 
&= \left(k^p+k+o(1)\right)\binom{n-2}{k-2}^p.
\end{align*}

For $2 \leq j \leq k-1$, the $j$-sets with $1\in J$ and $J \cap [2,k+1] \neq \emptyset$ provide the leading order term of the $(j,p)$-norm of $\cH(n,k)$. Therefore
\begin{align*}
\cojp(\cH(n,k)) &= \left(\binom{n-1}{j-1} - \binom{n-k-1}{j-1}\right)\binom{n-j}{k-j}^p + O(n^{j-1 + p(k-j-1) }) \\
&= \left(\frac{k}{(j-2)!(k-j)!^p} + o(1)\right)n^{j-2+p(k-j)}.
\end{align*}
\end{proof}

We now give the short proof of Hilton-Milner in the $(j, p)$-norm assuming Theorem~\ref{thm:hmjpnormlarge}.

\begin{proof}[Proof of Theorem~\ref{thm:hmjp}]
If $2\le j\le k-1$, then Theorem~\ref{thm:hmjp} follows by comparing the leading order terms in Lemma~\ref{lem:mjpcalc} and Theorem~\ref{thm:hmjpnormlarge}. 

If $j=1$, then for $k\ge 4$, we have $\max\{(k-1)^p + k -1, 3\cdot 2^p\} < k^{p} + k$, so again Lemma~\ref{lem:mjpcalc} and Theorem~\ref{thm:hmjpnormlarge} imply Theorem~\ref{thm:hmjp}. 
\end{proof}

The rest of this section is devoted to proving Theorem~\ref{thm:hmjpnormlarge}. We split the proof up according to whether $2\le j\le k-1$ or $j=1$. 

\subsection{The cases \texorpdfstring{$2\le j\le k-1$}{}: Bounding by \texorpdfstring{$j$-covers}{}}

Let $\cF \subset \binom{[n]}{k}$ be an intersecting family. We may assume that $\cF$ is an inclusion-maximal intersecting family, as the $(j, p)$-norms are strictly monotonic under edge addition. Recall that a set $T \subset [n]$ is a \emph{cover} of $\cF$ if $T\cap F\neq \emptyset$ for every $F\in \cF$. Define the family of $j$-covers of $\cF$ as 
\[\cC_j := \left\{J \in \binom{[n]}{j}: J\text{ is a cover of }\cF \right\}.\]

 We generalize our arguments from the proof of Theorem~\ref{thm:completeHMstability}. We first show that the family $\cC_j$ is intersecting.

\begin{lemma}\label{lem:cjint}
For $n > 2k$ and $2\le j\le k-1$, the family $\cC_j$ is an intersecting $j$-uniform family.
\end{lemma}

\begin{proof}
    If $|\cC_{j}| \le 1$, then clearly $\cC_j$ is intersecting, so we may assume that $|\cC_j| \ge 2$. Assume for a contradiction that $\cC_j$ is not intersecting. Let $I, J\in \cC_j$ such that $I\cap J = \emptyset$. Since $n > 2k$, we may choose disjoint $(k-j)$-sets $S, T \in [n] \setminus (I\cup J)$. The sets $I\cup S, J\cup T $ intersect every $F\in \cF$, so the maximality of $\cF$ implies $I \cup S, J\cup T \in \cF$, which is a contradiction because $\cF$ is intersecting.
\end{proof}

We next prove a lemma showing that elements of $\cC_j$ have full codegree in $\cF$, while $j$-sets not contained in $\cC_j$ have asymptotically smaller codegree.
\begin{lemma}\label{lem:jcover}
Let $\cF \subset \binom{[n]}{k}$ be a maximal intersecting family and let $j$ be an integer with $2\le j\le k-1$.  For each set $J \in \binom{[n]}{j}$:
\begin{enumerate}
\item If $J \in \cC_j(\cF)$, then $d(J) = \binom{n-j}{k-j}$.
\item If $J \notin \cC_j(\cF)$, then $d(J) \le k\binom{n-j-1}{k-j-1}$.
\end{enumerate}
\end{lemma}

\begin{proof}
  Let $J \subset [n]$ be a $j$-set. Suppose $J \in \cC_j$ is a cover of $\cF$. Then every $k$-set containing $J$ intersects every set in $\cF$, and so must be in $\cF$ by maximality. If, on the contrary, $J$ is not a cover, then there is some set $F_0 \in \cF$ such that $J \cap F_0 = \emptyset$. Every $F\in \cF$ which contains $J$ must have a nonempty intersection with $F_0$, so $d(J) \le k \binom{n-j-1}{k-j-1}$. 
\end{proof}

\begin{lemma}\label{lem:nonhmbd}
Let $k\ge 4$ and $2\le j\le k-1$. Suppose that $\cF \subset \binom{[n]}{k}$ is a maximal nontrivial intersecting family which is not a subfamily of any family isomorphic to $\cH(n, k)$. Then, for sufficiently large $n$ we have
\[|\cC_j(\cF)| \le \begin{cases}\frac{k-1}{(j-2)!}n^{j-2} + O_{k, j}(n^{j-3}) & \text{if } j\ge 3\\ k-1 & \text{if } j=2.\end{cases}\]
\end{lemma}

\begin{proof}
    By Lemma~\ref{lem:cjint}, the family $\cC_j$ is intersecting. First, suppose that $\cC_j$ is nontrivial intersecting. If $j=2$, then $\cC_j$ is contained in a triangle and thus $|\cC_j| \le 3\le k-1$. If $j\ge 3$, then by the Hilton-Milner theorem,
    \[|\cC_j| \le \binom{n-1}{j-1} - \binom{n-j-1}{j-1} + 1 = \frac{j}{(j-2)!}n^{j-2}+O_j(n^{j-3}).\]
    Suppose now that $\cC_j$ has an element $i$ which is contained in all sets in $\cC_j$. Since $\cF$ is nontrivial and not a subfamily of $\cH(n, k)$, there exist distinct $F, F'\in \cF$ such that $i\notin F\cup F'$. Since $J$ is a cover of $\cF$, we have that $(J-i)\cap F\neq\emptyset$ and $(J-i)\cap F'\neq\emptyset$. We bound $|\cC_j|$ by counting the number of possible sets in $\binom{[n]\setminus \{i\}}{j-1}$ which intersect both $F$ and $F'$. If $j=2$, then $(J-i) \subset F\cap F'$, so $|\cC_j(\cF)| \le |F\cap F'| \le k-1$. If $j\ge 3$, then we obtain 
    \begin{equation}\label{eq:cj-1bd}
|\cC_{j}| \le \binom{n-1}{j-1} - 2\binom{n-k-1}{j-1} + \binom{n-2k + |F\cap F'| - 1}{j-1}.
\end{equation}
    Since $1\le |F\cap F'| \le k-1$ and the upper bound in \eqref{eq:cj-1bd} is increasing in $|F\cap F'|$, the maximum occurs when $|F\cap F'| = k-1$. Then, we obtain
    \begin{align*}
    |\cC_j| &\le \binom{n-1}{j-1} - 2\binom{n-k-1}{j-1} + \binom{n-k-2}{j-1}\\
    &=\frac{k-1}{(j-2)!}n^{j-2} + O_{k, j}(n^{j-3}).
    \end{align*}
    Since $\max\{3, j\} \le k-1$ whenever $2\le j\le k-1$ and $k\ge 4$, the lemma follows.
\end{proof}

We can now prove Theorem~\ref{thm:hmjpnormlarge} when $2\le j\le k-1$. 

\begin{proof}[Proof of Theorem~\ref{thm:hmjpnormlarge} when $2\le j\le k-1$]
Let $\cF \subset \binom{[n]}{k}$ be a nontrivial intersecting family. By the Hilton-Milner theorem, $|\cF| \le \binom{n-1}{k-1} - \binom{n-k-1}{k-1} + 1 = O_k(n^{k-2})$. Using Lemmas~\ref{lem:jcover} and \ref{lem:nonhmbd}, we obtain
\begin{align*}
    \cojp(\cF) &= \sum_{J\in \binom{[n]}{j}}d(J)^p\\
    &=\sum_{J\in \cC_j(\cF)}d(J)^p + \sum_{J\notin\cC_j(\cF)}d(J)^p\\
    &\le \left(\frac{k-1}{(j-2)!}n^{j-2}+O_{k, j}(n^{j-3})\right)\binom{n-j}{k-j}^p + \sum_{J\notin \cC_j(\cF)}d(J)^{p-1}d(J)\\
    &\le \left(\frac{k-1}{(j-2)!}n^{j-2}+O_{k, j}(n^{j-3})\right)\binom{n-j}{k-j}^p + \left(k\binom{n-j-1}{k-j-1}\right)^{p-1}\binom{k}{j}|\cF|\\
    &= \left(\frac{k-1}{(j-2)!}n^{j-2}+O_{k, j}(n^{j-3})\right)\binom{n-j}{k-j}^p + O_{j, k, p}(n^{k-2 + (p-1)(k-j-1)})\\
    &=\left(\frac{k-1}{(j-2)!(k-j)!^p}+o(1)\right)n^{j-2+p(k-j)}. 
\end{align*}
\end{proof}

\subsection{The degree case \texorpdfstring{$j=1$}{}}

We now handle the $j=1$ case of Theorem~\ref{thm:hmjpnormlarge}. We can assume that $\cF$ is a maximal intersecting family. Let $G = \cC_2(\cF)$ be the graph whose edges are the $2$-covers of $\cF$. By Lemma~\ref{lem:jcover}, we have that
\[d_{\cF}(\{v, w\}) = \begin{cases}\binom{n-2}{k-2}, & \text{if } \{v, w\} \in E(G) \\ \le k\binom{n-3}{k-3} & \text{if } \{v, w\} \notin E(G).\end{cases}\]

By Lemma~\ref{lem:cjint}, we know that $G$ is intersecting, and so it is either a star or a triangle. If $G$ is a triangle, then it has three vertices each of degree two. Otherwise, if $G$ is a star, then observe that $d_G(u) \le k$ if $\cF$ is nontrivial, and $d_G(u) \le k-1$ if $\cF$ is nontrivial and not a subfamily of $\cH(n, k)$. Indeed, if $d_G(u) = k$, then every set in $\cF$ which avoids $u$ must contain all $k$ of the neighbors of $u$, so nontriviality implies $\cF$ is contained in a family isomorphic to $\cH(n, k)$. 

We now prove a lemma showing that for each $v\in [n]$, its degree in $\cF$ is asymptotically captured by its degree in $G$. 

\begin{lemma}\label{lem:vtxdeg}
Let $k\ge 4$, $n$ sufficiently large, and let $\cF \subset \binom{[n]}{k}$ be a maximal nontrivial intersecting family. For each $v\in [n]$, 
\begin{equation}\label{eqn:degvtx}
d_{\cF}(v) = d_G(v)\binom{n-2}{k-2} + O_{k}(n^{k-3}).
\end{equation}
For each real number $p > 1$, 
\[\coonep(\cF) = \left(\sum_{v\in [n]}d_G(v)^p+o(1)\right)\binom{n-2}{k-2}^p.\]
\end{lemma}

\begin{proof}
    Let $v\in [n]$. Since $\cF$ is nontrivial, there is a set $B_v \in \cF$ which does not contain $v$. Every set $A\in \cF$ which contains $v$ must also intersect $B_v$, so
    \[d_{\cF}(v)\le \sum_{w\in B_v}d_{\cF}(\{v, w\}) = \sum_{\substack{w \in B_v \\ vw\in E(G)}}d_{\cF}(\{v, w\}) + \sum_{\substack{w\in B_v \\ vw\notin E(G)}}d_{\cF}(\{v, w\}) \le d_{G}(v)\binom{n-2}{k-2} + k^2\binom{n-3}{k-3}.\]
    On the other hand, any two edges $vw, vu \in E(G)$ are both contained in at most $\binom{n-3}{k-3}$ sets in $\cF$. Hence, by inclusion-exclusion, 
    \[d_{\cF}(v) \ge d_{G}(v) \binom{n-2}{k-2} - \binom{d_{G}(v)}{2}\binom{n-3}{k-3}.\]
    Since $d_G(v) \le k$ for each $v\in [n]$, we have $d_{\cF}(v) \ge d_{G}(v) \binom{n-2}{k-2} - \binom{k}{2}\binom{n-3}{k-3}$, proving \eqref{eqn:degvtx}. 
    Since $G$ is either a star with at most $k$ leaves or a triangle, it follows that at most $k+1$ vertices in $[n]$ have $d_G(v) \ge 1$. For the vertices $v\in [n]$ which have $d_{G}(v) = 0$, \eqref{eqn:degvtx} implies that $d_{\cF}(v) = O_k(n^{k-3})$. Recall that the Hilton-Milner theorem implies $\sum_{v}d_{F}(v) = k|\cF| = O_k(n^{k-2})$. Hence, we compute the $(1, p)$-norm of $\cF$ as
    \begin{align*}
        \coonep(\cF) &= \sum_{v\in [n]}d_{\cF}(v)^p \\
        &=\sum_{\substack{v\in [n] \\ d_{G}(v)\ge 1}}d_{\cF}(v)^p + \sum_{\substack{v\in [n] \\ d_{G}(v) = 0}}d_{\cF}(v)^p\\
        &=\sum_{\substack{v\in [n] \\ d_{G}(v)\ge 1}}\left(d_G(v)^p + o(1)\right)\binom{n-2}{k-2}^p + \sum_{\substack{v\in [n] \\ d_{G}(v) = 0}}d_{\cF}(v)^{p-1}d_{\cF}(v)\\
        &\le \sum_{\substack{v\in [n] \\ d_{G}(v)\ge 1}}\left(d_G(v)^p + o(1)\right)\binom{n-2}{k-2}^p + \sum_{v\in [n]}d_{\cF}(v)^{p-1}d_{\cF}(v)\\
        & \le \sum_{\substack{v\in [n] \\ d_{G}(v)\ge 1}}\left(d_G(v)^p + o(1)\right)\binom{n-2}{k-2}^p + \left(k^2\binom{n-3}{k-3}\right)^{p-1}k|\cF|\\
        & = \sum_{\substack{v\in [n] \\ d_{G}(v)\ge 1}}\left(d_G(v)^p + o(1)\right)\binom{n-2}{k-2}^p + o\left(n^{(k-2)p}\right)\\
        & = \left(\sum_{v \in [n]}d_{G}(v)^p+o(1)\right)\binom{n-2}{k-2}^p.
    \end{align*}
    This completes the proof. 
\end{proof}

We now complete the proof of Theorem~\ref{thm:hmjpnormlarge}. 

\begin{proof}[Proof of Theorem~\ref{thm:hmjpnormlarge} when $j=1$]
    As the $(j, p)$-norm does not decrease under edge addition, we may suppose $\cF$ is a maximal nontrivial intersecting family with $\cF \not \subset\cH(n, k)$. Then the two-cover graph $G = \cC_2(\cF)$ is either a star with $r \le k-1$ edges or a triangle. If $G$ is a star, then $\sum_{v\in [n]}d_G(v)^p = r^p + r \le (k-1)^p + k-1$. If $G$ is a triangle, then $\sum_{v\in [n]}d_G(v)^p = 3\cdot 2^p$. By Lemma~\ref{lem:vtxdeg},
    \[\coonep(\cF) \le( \max\{(k-1)^p + k -1, 3\cdot 2^p\} + o(1))\binom{n-2}{k-2}^p.\]
\end{proof}
\section{Exact 3-uniform Hilton-Milner in \texorpdfstring{$(j, p)$-norm}{}}
In this section, we prove an exact Hilton-Milner result in the $(j,p)$-norm for nontrivial 3-uniform families.
\begin{proof}[Proof of Proposition~\ref{thm:jpk=3}]
We directly compute 
\begin{align*}
\coonep(\cH(n, 3)) &= (3n-9)^p + 3(n-1)^p + (n-4)3^p \\
\coonep(\cA(n, 3)) &= 3(2n-5)^p + (n-3)3^p\\
\cotwop(\cH(n, 3)) &= 3(n-2)^p + (n-4)3^p + 3\cdot 2^p + 3(n-4)\\
\cotwop(\cA(n, 3)) &= 3(n-2)^p + 3(n-3)2^p.
\end{align*}

Note that
\[\cotwop(\cH(n, 3)) - \cotwop(\cA(n, 3)) = (n-4)(3^p - 3\cdot 2^p + 3) = 3(n-4)(3^{p-1} + 1 - 2\cdot 2^{p-1}).\]
When $p=2$, this expression is $0$. When $1 < p < 2$, the function $g(t) = t^{p-1}$ is strictly concave, so $\cotwop(\cH(n, 3)) - \cotwop(\cA(n, 3)) < 0$. When $p > 2$, the function $g(t) = t^{p-1}$ is strictly convex, so $\cotwop(\cH(n, 3)) - \cotwop(\cA(n, 3)) > 0$. 

For the $(1, p)$-norm, 
\[\coonep(\cH(n, 3)) - \coonep(\cA(n, 3)) = (3n-9)^p + 3(n-1)^p - 3(2n-5)^p - 3^p. \]
When $p=2$, this evaluates to $0$. Otherwise, define the function $f(x) = (3x-9)^p + 3(x-1)^p-3(2x-5)^p-3^p$. Then $f(4) = 3^p + 3\cdot3^p-3\cdot 3^p - 3^p = 0$. Observe that $f'(x) = 3p((3x-9)^{p-1} + (x-1)^{p-1}-2(2x-5)^{p-1})$. Since $2x-5 = \frac{(x-1)+(3x-9)}{2}$, it follows from the concavity of $t^{p-1}$ that when $1 < p < 2$, $\coonep(\cH(n, 3)) - \coonep(\cA(n, 3)) < 0$. The convexity of $t^{p-1}$ when $p > 2$ similarly implies that $\coonep(\cH(n, 3)) - \coonep(\cA(n, 3)) > 0$ in that case. 

As the $(j, p)$-norm is monotonic under edge addition, we may assume that any other nontrivial intersecting $\cF$ is inclusion-maximal. Polcyn and Ruci\'nski~\cite{PR} classified all of the maximal $3$-uniform intersecting families. There are fourteen such non-star families. It can be checked directly that all twelve of the families which are not isomorphic to $\cH(n, 3)$ or $\cA(n, 3)$ have smaller $(1, p)$-norm and $(2, p)$-norm than $\cH(n, 3)$. 
\end{proof}
\section*{Acknowledgment}
ChatGPT was used as an aid in the development of the paper. All arguments in the paper have been independently written and verified by the authors, who take full responsibility for the contents of the paper.

\end{document}